\documentclass[12pt]{amsart}
\usepackage{graphicx} 

\usepackage{soul} 
\usepackage{amscd,amssymb,mathtools}
\usepackage[arrow,matrix,graph,frame,poly,arc,tips]{xy}
\usepackage[dvipsnames]{xcolor}
\usepackage{graphicx}
\usepackage{calrsfs}
\usepackage[labelfont=rm]{subcaption}
\usepackage{tikz-cd} 
\usepackage{tikz}
\usetikzlibrary{decorations.markings}
\usetikzlibrary{shapes}
\usetikzlibrary{backgrounds}
\usetikzlibrary{calc}
\usetikzlibrary{shapes.misc, positioning}
\usepackage{mdwlist}
\usepackage{xfrac}
\DeclareCaptionSubType*{figure}

\usepackage{multirow}
\usepackage{enumerate}
\usepackage{verbatim}
\usepackage{comment}
\usepackage{todonotes}

\DeclarePairedDelimiter{\form}{\langle}{\rangle}

\newcommand\on[1]{\operatorname{#1}}

    \newcommand\ba{\begin{align*}}
    \newcommand\ea{\end{align*}}
    \newcommand\be{\begin{enumerate}}
    \newcommand\ee{\end{enumerate}}
    \newcommand\bpf{\begin{proof}}
    \newcommand\epf{\end{proof}}
    \newcommand\bpp{\begin{prop}}
    \newcommand\epp{\end{prop}}
    \newcommand\bpb{\begin{prob}}
    \newcommand\epb{\end{prob}}
    \newcommand\bd{\begin{defn}}
    \newcommand\ed{\end{defn}}
    \newcommand\bh{\begin{hint}}
    \newcommand\eh{\end{hint}}

  \newcommand{\flows}[0]{\mathrm{flows}}
  \newcommand{\param}[0]{\mathrm{param}}

  \newcommand{\MMc}{\mathcal{M}^c_{(\mathrm{vol})}}
  \newcommand{\MMv}{\mathcal{M}_{(\mathrm{vol})}}
   \newcommand{\CLM}{\mathcal{C}_{\mathrm{LM}}}

    \newcommand\bN{\mathbb{N}}
    \newcommand\N{\mathbb{N}}
    \newcommand\bR{\mathbb{R}}
    \newcommand\bQ{\mathbb{Q}}
    
    \newcommand\Z{\mathbb{Z}}

    \newcommand\CC{\mathcal{C}}
    \newcommand\EE{\mathcal{E}}

    \newcommand\DD{\mathcal{D}}

    \newcommand\MM{\mathcal{M}}

    \newcommand{\R}[0]{\mathbb{R}}

\DeclareMathOperator\AGAPE{AGAPE}
\newcommand\vol{\operatorname{vol}}

    \newcommand\supp{\operatorname{supp}}
    \newcommand\Id{\operatorname{Id}}

    \DeclareMathOperator\Homeo{Homeo}

    \newcommand\sse{\subseteq}
    \newcommand\co{\colon}
        
    \DeclareMathOperator\im{im}
    
    \DeclareMathOperator\fix{fix}

    \DeclareMathOperator\Diff{Diff}

    \DeclareMathOperator\RO{RO}
    \DeclareMathOperator\inte{int}
    \DeclareMathOperator\cl{cl}

    \DeclareMathOperator\appl{appl}
    \DeclareMathOperator\Th{{Th}}
    \DeclareMathOperator\ACT{{ACT}}

\newcommand\suppe{\operatorname{{supp}^{\mathrm{e}}}}

    \def\thetitle{Elementary equivalence and diffeomorphism groups of smooth manifolds}
    \def\theauthors{{J. de la Nuez Gonz\'alez, Sang-hyun Kim, Thomas Koberda}}
    \usepackage{hyperref}
    \hypersetup{
    	colorlinks=false,
    	plainpages,
    	urlcolor=black,
    	linkcolor=black
    	pdftitle=   \thetitle,
    	pdfauthor=  {\theauthors}
    }

    \theoremstyle{plain}
    \newtheorem*{mainthm}{Main Theorem}

    \newtheorem{thm}{Theorem}[section]
    \newtheorem{lem}[thm]{Lemma}
    \newtheorem{lemma}[thm]{Lemma}
    \newtheorem{cor}[thm]{Corollary}
    \newtheorem{prop}[thm]{Proposition}

    \newtheorem*{claim*}{Claim}

    \newtheorem{thmA}{Theorem}

    \theoremstyle{remark}
    
    \newtheorem{rem}[thm]{Remark}

    \theoremstyle{definition}
    \newtheorem{defn}[thm]{Definition}
    \newtheorem{prob}{Problem}[section]

\begin{document}
    \title[Elementary equivalence and diffeomorphism groups]\thetitle
    \date{\today}

\title\thetitle
    \date{\today}
    \author[S. Kim]{Sang-hyun Kim}
    \address{School of Mathematics, Korea Institute for Advanced Study (KIAS), Seoul, 02455, Korea}
    \email{skim.math@gmail.com}
    \urladdr{https://www.kimsh.kr}

    \author[T. Koberda]{Thomas Koberda}
    \address{Department of Mathematics, University of Virginia, Charlottesville, VA 22904-4137, USA}
    \email{thomas.koberda@gmail.com}
    \urladdr{https://sites.google.com/view/koberdat}
    
    \author[J. de la Nuez Gonz\'alez]{J. de la Nuez Gonz\'alez}
    \address{School of Mathematics, Korea Institute for Advanced Study (KIAS), Seoul, 02455, Korea}
    \email{jnuezgonzalez@gmail.com}

\begin{abstract}
We interpret smooth structures of manifolds by the first order theory of manifold diffeomorphism groups. Let $M$ and $N$ be smooth manifolds, with $M$ closed and connected. If the $C^r$--diffeomorphism group of $M$ is elementarily equivalent to the $C^s$--diffeomorphism group of $N$ for some $r,s\in[1,\infty]$, then $r=s$ and $M$ and $N$ are diffeomorphic. This strengthens a previously known result by Takens and Filipkiewicz, which asserts that for integers $r$ and $s$, if $\Diff^r(M)$ is isomorphic to $\Diff^s(N)$ as abstract groups then the regularities $r$ and $s$ are the same and the underlying manifolds are diffeomorphic. We prove an analogous result for groups of diffeomorphisms preserving smooth volume forms, in dimension at least two. Along the way, we obtain a group sentence in $\Diff^r(M)$ expressing that there exists a manifold homeomorphic but not diffeomorphic to $M$.
\end{abstract}
    
    \maketitle
    
\setcounter{tocdepth}{1}
\tableofcontents
    

\section{Introduction}
In this paper, we initiate the study of first order rigidity phenomena among manifold diffeomorphism groups.

For a smooth (i.e. $C^\infty$) manifold $M$ and for $r\in[1,\infty]\cup\{0\}$, we let $\Diff^r(M)$ denote the $C^r$ diffeomorphism group of $M$. 
When $M$ is orientable and equipped with a smooth volume form $\omega$, we let $\Diff^r(M,\omega)$ denote the volume-preserving subgroup. We often make the volume form implicit and write $\Diff^r_{\vol}(M)$ instead. 
We denote by $\Diff^r_0(M)$ and $\Diff^r_{0,\vol}(M)$ the identity components of the corresponding groups. When $r=0$, we prefer to use the
notation $\Homeo$ instead of $\Diff^0$. For statements that apply equally to a full group of diffeomorphisms or to a group of volume-preserving diffeomorphisms, we will write $\Diff^r_{(\vol)}$ and $\Diff^r_{0(,\vol)}$ when either group can be substituted.

Recall that the \emph{theory} $\Th{G}$ of a group $G$ is the set of all the first-order sentences in the language of group theory (i.e. the set of finite sequences of identity, composition and inverse symbols as well as logical symbols and variables) that are valid in $G$. We say two groups are \emph{elementarily equivalent} and write $G\equiv H$ if $\Th{G}=\Th{H}$. The goal of this paper is to establish the following result on the first order theory of manifold diffeomorphism groups.

\begin{mainthm}\label{thm:main}
Suppose we have two smooth manifolds $M$ and $N$, such that $M$ is closed and connected.
For some $r,s\in\{0\}\cup[1,\infty]$, we assume one of the following.
\begin{enumerate}[(i)]
    \item Some elementarily equivalent groups $G$ and $H$ satisfy
$$\Diff^r_{0}(M)\leq G\leq \Diff^r(M)$$ and 
$$\Diff^s_{0}(N)\leq H\leq \Diff^s(N).$$
\item 
The manifolds $M$ and $N$ are orientable, of dimension at least two, equipped with smooth volume forms, and some elementarily equivalent groups $G$ and $H$ satisfy
$$\Diff^r_{0,\vol}(M)\leq G\leq \Diff^r_{\vol}(M)$$ and 
$$\Diff^s_{0,\vol}(N)\leq H\leq \Diff^s_{\vol}(N).$$
\end{enumerate} 
Then in each of the cases, we have that $r=s$ and that $M$ and $N$ are homeomorphic; furthermore, if $r,s\ge1$, then $M$ and $N$ are diffeomorphic.
\end{mainthm}

\begin{rem}\label{rem:topology}
    We do not consider the case that $0<r<1$, since $\Diff^r(M)$ is not closed under composition in general; see~\cite[Chapter A.2]{KK2021book}.
\end{rem}

It is known from the work of Takens, Bochner--Montgomery, and
Filipkiewicz~\cite{BM1945,Takens1979,filipkiewicz} that for integer regularities, two diffeomorphism groups of closed manifolds are isomorphic as abstract groups if and only if the regularities agree and the isomorphism is induced by a diffeomorphism of that regularity between the underlying manifolds. In fact, it is a consequence of the Generalized Takens--Filipkiewicz Theorem~\cite[Theorem 3.1.2]{KK2021book} that the Main Theorem holds with a stronger assumption that $G$ and $H$ are isomorphic groups (rather than merely
elementarily equivalent), and that $r,s\in
\mathbb{Z}_{+}\cup\{\infty\}$.
However, it was an open question as to whether or not the Takens--Filipkiewicz Theorem holds for non-integer regularities; the main difficulty of the question was the lack of a generalization of Bochner--Montgomery's result (Theorem~\ref{thm:bochner-montgomery}) to such regularities. Our Main Theorem implies a positive answer. 

\begin{cor}\label{cor:isomorphic}
Let $\MM_0$ be the class of pairs $(M,\Diff^r_0(M))$ such that $M$ is a closed, connected and smooth manifold, and such that $r\in[1,\infty]$. 
If two pairs $(M,\Diff^r_0(M))$ and $(N,\Diff^s_0(N))$ in the class $\MM_0$ satisfy $G\cong H$,
then $r$ is equal to $s$ and $M$ is diffeomorphic to $N$.
\end{cor}
It is worth noting that Corollary~\ref{cor:isomorphic} is an instance of a purely geometric topology result first proved via applied logic. The reader may find a further treatment of the algebraic structure of diffeomorphism groups in dimension one in~\cite{Navas2011,CKK2019,KKM2019} and the references therein, for example.

Main Theorem fits in the framework of research on the first order rigidity of groups. We say a class of structures $\CC$ is \emph{first order rigid} if whenever two structures $G$ and $H$ of $\CC$ are elementarily equivalent then $G$ and $H$ are isomorphic. The first order rigidity was established in various structures such as rings~\cite{leloup,greenfeld}, lattices in higher rank~\cite{avni-etal} and linear groups~\cite{plotkin}. The authors previously proved the first order rigidity of the class of pairs $(M,\Homeo(M))$, where $M$ is a compact topological manifold (possibly with boundary and possibly non-smooth)~\cite{kkdln2025}. Main Theorem immediately implies that the class $\MM_0$ as above is first order rigid.

\textit{From now on, we adopt the convention that all manifolds under consideration are orientable and of dimension at least two, whenever $\Diff^r_{\vol}$ and $\Diff^r_{0,\vol}$ are discussed.}
We prove the Main Theorem largely in two steps, as described by the following two theorems.
In the first step, we construct a group theoretic sentence expressing that ``My underlying manifold is diffeomorphic to $M$'':
\begin{thmA}[Manifold recognition]\label{thm:mfd}
For all closed, connected, smooth
manifolds $M$, there exists a sentence $\phi_M$ such that the following hold:
\begin{enumerate}[(i)]
\item\label{p:mfd1} if we have $\Diff^\infty_{0(,\vol)}(M)\le G\le \Diff^1_{(\vol)}(M)$, then $\phi_M\in\Th G$;
\item\label{p:mfd2} if a group $H$ satisfies $\Diff^\infty_{0(,\vol)}(N)\le H\le \Diff^1_{(\vol)}(N)$
for some closed, connected, smooth
manifold $N$
and if $\phi_M\in\Th H$, then $M$ and $N$ are diffeomorphic.
\end{enumerate}
\end{thmA}
Note that the sentence $\phi_M$ may differ in the  volume-preserving case. In the second step, we find a formula $\rho_q(g)$ implying (but not necessarily equivalent to) the statement that ``the diffeomorphsim $g$ is not $C^s$ for $s>q$''.

\begin{thmA}[Regularity detection]\label{thm:reg}
For a closed, connected, smooth manifold $M$, and for a rational number $q\ge1$, there exists a formula $\rho_q(x)=\rho^{(\vol)}_{M,q}(x)$ such that whenever a group $G$ satisfies
\[
\Diff^\infty_{0(,\vol)}(M)\le G\le\Diff^1_{(\vol)}(M),\]
the following conclusions hold:
\begin{enumerate}[(i)]
\item\label{p:reg1} If $\Diff^q_{0(,\vol)}(M)\le G$, then $\rho_q(g)\in\Th G$ for some $g\in G$;
\item\label{p:reg2} If $\rho_q(g)\in \Th G$ for some $g\in G$, then $g$ does not belong to $\bigcup_{s>q}\Diff^s(M)$.
\end{enumerate}
\end{thmA}

The Main Theorem will then be a straightforward consequence modulo consideration on some exceptional cases. Let us make the extra assumptions that all manifolds are closed and connected with dimension at least two and that $r,s\ge1$. Then indeed, we see from Theorem~\ref{thm:mfd}, combined the hypothesis of Main Theorem, that $M$ and $N$ are diffeomorphic. If $1\le r<s\le\infty $, then let us pick $q\in[r,s)\cap\mathbb{Q}$. Theorem~\ref{thm:reg} will imply that $\rho_{q}(g)$ holds in $G$ for some $g\in G$, while it holds for no $g\in H$; this contradicts the elementary equivalence.

Along the way of proving Theorem~\ref{thm:mfd}, we also observe that the first order theory of sufficiently large manifold homeomorphism groups can detect the homeomorphism types of the manifold.

\begin{thm}\label{thm:homeomorphism type-intro} (Theorem~\ref{thm:homeomorphism type})
    Let $M$ be a connected, closed, smooth manifold,
    and let \[\Diff^\infty_{0(,\vol)}(M)\leq G\leq\Homeo_{(\vol)}(M),\] There is a
    sentence $\phi_{M,0}\in \Th{G}$ such that for all smooth (not necessarily closed
    or even compact) manifolds $N$ and
    \[\Diff^\infty_{0(,\vol)}(N)\leq H\leq\Homeo_{(\vol)}(N),\]
    we have $\phi_{M,0}\in\Th{H}$ if and only if $M$ and $N$ are
    homeomorphic.
\end{thm}

From the first order theory of diffeomorphism groups,
we will later be able to interpret differentiable maps between Euclidean spaces, and hence smooth manifolds embedded in Euclidean spaces. This observation, combined with Theorems~\ref{thm:mfd} and~\ref{thm:homeomorphism type-intro} lets us express the existence of exotic smooth structures on any particular underlying manifold.

\begin{cor}\label{cor:exotic}
For a given closed, connected, smooth
manifold $M$
there exists a sentence $\epsilon_M$ such that
if $ \epsilon_M\in \Th G$ 
for some  group $G$ satisfying
\[\Diff^\infty_{0(,\vol)}(M)\leq G\leq \Diff^1_{(\vol)}(M),\]
then there is a manifold that is homeomorphic to $M$ but not diffeomorphic to $M$.
\end{cor}

As an application of Corollary~\ref{cor:exotic}, there is a sentence which asserts ``there exists 
an exotic $4$--sphere". Verifying the validity of the sentence $\epsilon_M$ in a
diffeomorphism group appears practically impossible; see~\cite{KdlN23}.

Finally, we record the following curiosity.
See Proposition~\ref{prop:compact-noncompact} and paragraphs following it for details.

\begin{cor}\label{cor:compact-noncompact-intro}
There is a sentence $\on{compact}$
such that for all smooth boundaryless manifold $M$ 
the following are equivalent:
\begin{enumerate}
    \item $M$ is compact.
    \item We have a sentence $\on{compact}\in \Th G$ for some
\[\Diff^\infty_0(M)\le G\le\Diff^1(M),\]
or for some $G$ satisfying
\[ \Homeo_0(M)\le G\le\Homeo(M).\]
\item  
We have a sentence $\on{compact}\in \Th G$ for all
\[\Diff^\infty_0(M)\le G\le\Diff^1(M),\]
and for all $G$ satisfying
\[ \Homeo_0(M)\le G\le\Homeo(M).\]
\end{enumerate}
\end{cor}
An analogous result holds for homeomorphisms, measure preserving homeomorphisms and volume-preserving diffeomorphisms.

This paper is structured as follows. In Section~\ref{sec:agape}, we interpret a seemingly much more complicated structure $\AGAPE(M,G)$
within the group structure $G$. In Section~\ref{s:mfd}, we use this former structure to detect whether or not two given smooth manifolds are diffeomorphic to each other. A smooth variation of a lemma by Cheeger and Kister is a crucial ingredient here. By the interpretability, we conclude that the group $G$ can detect the smooth structure of $M$.
In Section~\ref{s:reg}, we also use this structure $\AGAPE(M,G)$ to express that the group $G$ contains a diffeomorphism of a certain precise regularity. Harrison's diffeormorphism furnishes the desired
diffeomorphism. Using the interpretability again, we conclude that $G$ detects the regularity of the ambient diffeomorphism group. 
In the last section, we combine the previous sections and also consider some exceptional cases that need separate attention, finally giving a proof of Main Theorem.

\section{Three interpretability results}\label{sec:agape}
From this point on, we consider the class $\MM$ of pairs $(M,G)$ such that $M$ is a smooth manifold and $G$ is a group such that 
$$\Diff_0^\infty(M)\le G\le \Homeo(M).$$
We also consider the class $\MM_{\vol}$ of pairs $(M,G)$ such that $M$ is a smooth oriented manifold of dimension at least two with a smooth volume form $\omega$, and $G$ satisfies
$$\Diff_{0,\vol}^\infty(M,\omega)\le G\le \Homeo_{\vol}(M,\omega).$$
We may assume that $\omega$ is the Riemannian volume form for some fixed Riemanninan metric on $M$, after a conformal change.

Our strategy for Main Theorem can be roughly divided into two steps. In the first step, we investigate the expressive power of the group language of $G$ for $(M,G)\in\MMv$. Such a language is powerful enough to generate expressions for second order arithmetic (i.e. analysis) and smooth embeddings of (sufficiently many) balls into $M$. We use this in the second step, where we formulate the diffeomorphism types of $M$ and regularities of $G$ in terms of those expressions.

In this section, we focus on the first step. Namely, we will build seemingly complicated and detailed languages (called $\ACT$ and $\AGAPE$) from the group language. The structures corresponding to these new languages are shown to be powerful enough for our purpose of expressing smooth balls. On the other hand, we will ``interpret'' these new structures back to group structures and conclude that the original structure, namely the group itself, is already powerful.
Precise mathematical foundations of this strategy will be laid down in this section. 

This approach is in parallel with our previous paper~\cite{kkdln2025}; however, we reiterate that the end goal of this paper is the interpretation of smooth structures on manifolds, which is much more subtle and goes far beyond the known literature~\cite{Rubin1989, Rubin1996} on the interpretability of $C^0$--structures. This problem implicitly necessitates distinguishing between homeomorphic manifolds which are not diffeomorphic, such as exotic spheres.

\subsection{Preliminaries on differential topology}
Let us write $B(k)$ for the compact ball of radius $k$ centered at the origin in $\bR^n$, where the dimension $n$ is implicit. A \emph{embedded collared $C^k$-ball}, or simply a $C^k$-\emph{collared ball} in $M$ is a $C^k$--embedding 
$$\phi\colon B(1)\longrightarrow M$$
that extends to a $C^k$ embedding of $B(2)$. We call $\phi(B(2))\setminus \phi(B(1))$ a
\emph{collar} of the ball. The image of the embedding $\phi$ is often referred to as a ball as well.
We will say two collared balls
$$\phi_1,\phi_2\colon B(1)\longrightarrow M$$
have \emph{disjoint collars} if there exist extensions of $\phi_1$ and $\phi_2$ to $B(2)$ with disjoint images.
We suppress $C^k$ from the notation when the meaning is clear from
context.

If $n=1$ then collared balls are identical with open intervals whose closures not equal to the entirety of $M$. In this case, if $\{B_1,\ldots,B_k\}$ are embedded collared balls in $M$ with disjoint collars, we say that $B_i$ and $B_{i+1}$ are \emph{adjacent} (in the natural linear or cyclic order on $M$) if there is a path from $B_i$ to $B_{i+1}$ which does not meet the interior of $B_j$ for any $j\neq i,i+1$. We will say that $\{B_1,\ldots,B_k\}$ is \emph{ordered} if $B_i$ is adjacent to $B_{i+1}$ for $1\leq i\leq k-1$. There is an obvious notion of an \emph{order preserving} bijection between collections of balls $\{B_1,\ldots,B_k\}$
and $\{B_1',\ldots,B_k'\}$, and we say these collections are \emph{order isomorphic} if such a bijection exists. It is trivial to see that in dimension one, for $k$ arbitrary and two collections $\{B_1,\ldots,B_k\}$ and $\{B_1',\ldots,B_k'\}$ of collared balls with disjoint collars, there exists a smooth diffeomorphism $\phi$ with $B_i\mapsto B_i'$ for all $i$ if and only
if there is an order preserving bijection between $\{B_1,\ldots,B_k\}$
and $\{B_1',\ldots,B_k'\}$.

\begin{lemma}\label{lem:cr-trans-balls-gen}
Let $M$ be a smooth connected manifold.
\begin{enumerate}
    \item If $\dim M=1$, then $\Diff_0^\infty(M)$ acts transitively on order-isomorphic ordered
    $k$--tuples of collared balls with disjoint collars.
    \item\label{p:transitive}  If $\dim M\geq 2$, 
    and if $(U_1,\ldots,U_k)$ and $(V_1,\ldots,V_k)$ are two $k$--tuples of collared balls with disjoint collars 
    in $M$,
    then there exists some $g\in \Diff_0^\infty(M)$ such that 
    $g(U_i)= V_i$ for all $i$.
    \item The group $\Diff_{0,\vol}^\infty(M)$ acts transitively on collared balls of the same volume.
\end{enumerate}
\end{lemma}

\begin{proof}
We have argued the case where $\dim M=1$ already. In higher dimension, observe first
that after applying a $C^{\infty}$ diffeomorphism of $M$,
we may reduce to the case where $\{B_1,\ldots,B_k\}$
and $\{B_1',\ldots,B_k'\}$ lie in a single smooth chart of $M$. The lemma then follows
from 
a straightforward application of smooth isotopy extension theorem; see~\cite[Chapter 8]{hirsch2012differential}).
    The proof of the last part is identical to
    Corollary 2.2 in~\cite{LeRoux2014}, replacing the
    Oxtoby--Ulam Theorem 
by Moser's Theorem 
and using the previous parts
    for transitivity on smooth balls; cf.~\cite{kkdln2025}.
\end{proof}

The $C^0$--version of the following is Lemma 2.5 in~\cite{kkdln2025}. The proof for
smooth manifolds is more intuitive.

\begin{lemma}\label{lem:trans-measure}
If $U\subseteq M$ is a connected open set of volume $A>0$, then for all $0<a<A$ there exists a (smooth) collared ball $B$ of volume $a$, embedded in $U$. Moreover, if $\dim M
\geq 2$, then we may
arrange for $U\setminus B$ to be connected.
\end{lemma}

\begin{proof}
Since a smooth volume form defines a Radon measure, we can find a finite union $V$ of smooth balls with disjoint collars in $U$ such that $\vol(V)>a$. We inductively choose disjoint smooth arcs $\gamma_1,\ldots,\gamma_{k}$ with endpoints in $V$ so that the following hold for each $i$: (i) the complement of $X_{i}=V\cup\gamma_1\cup\cdots\gamma_i$ stays connected; (ii) each connected component of $X_i$ is simply connected; (iii) when $i<k$, the set $X_{i+1}$ has one less connected component than $X_i$. Using the tubular neighborhood theorem to each ball, and also to $\gamma_i$ chosen to transversely intersect the boundary of each ball, we obtain a smoothly embedded tube $T_i$ around $\gamma_i$ so that the resulting union $X=V\cup\bigcup_i T_i $ is a piecewise smooth embedding of a ball with volume larger than $a$. As a codimension-0 submanifold, It follows that $X$ is arbitrarily $C^0$-close to a smoothly embedded ball $Y$; see~\cite{Hirsch1962} and~\cite[Theorem 4.16]{FNOP2024}. By requiring $\vol(Y)>a$, and taking a suitable smaller concentric ball if necessary, we obtain a desired (smoothly embedded) collared ball of the exact volume.
\end{proof}

\subsection{Action structures of locally moving groups}\label{ss:act}
Let us briefly review key concepts from model theory that will be used in this paper. The reader may refer to \cite[Section 2.3]{kkdln2025}; more detailed accounts can be found in~\cite{marker-book,tent-ziegler}, for example.

Briefly speaking, the model theory deals with the \emph{sentences} of a \emph{language} for a mathematical \emph{structure}. The precise definitions of the terms in the case of groups go as follows. Let $G$ be a group. The group language consists of finite sequences of \emph{logical symbols} ($\wedge, \vee, \to, \neg, =$ and names $x_0,x_1,\ldots$ for variables)
and the group \emph{signature} (the identity symbol $1$, the inverse symbol ${}^{-1}$ and the composition symbol $\circ$, which will be used as $0$-, $1$- and $2$-ary operations). A (well-formed) \emph{formula} is a finite sequence that logically makes sense but is not necessarily true, such as $x_0=x_1^2$ or $x_1\ne x_1$. A \emph{sentence} is a formula without unquantified variables; by design, a sentence is either true or false in a specific structure, namely $G$. The \emph{universe} of the group structure $G$ is the set $G$ itself. The \emph{theory} $\Th{G}$ of a group is the set of all sentences that are valid in $G$. 
When $\phi\in\Th{G}$, we write $G\models \phi$ or simply, $\models\phi$ as long as the meaning is clear.
The \emph{universe} of a specific structure $G$ is the set $G$ itself.

Let us now expand the group language to the \emph{action} language. 
It will be convenient for us to consider a more general class of group actions. Let $X$ be a topological space. A subgroup $G$ of $\Homeo(X)$ is \emph{locally moving} if for each nonempty open set $U\subseteq X$ there exists a nontrivial $g\in G$ such that $g\restriction_{X\setminus U}=\Id$.
Let $\CLM$ denote the class of pairs $(X,G)$ where $X$ is a Hausdorff space and $G\le\Homeo(X)$ is locally moving.

Let $(X,G)\in\CLM$. 
For a subset $U\subseteq X$, we denote by $\inte U$ and $\cl U$ the interior and the closure of $U$, respectively.
Recall $U$ is \emph{regular open} if $U = \inte \cl U$. We let $\RO(X)$ be the collection of all the regular open subsets. 
The universe of the action structure $\ACT(X,G)$ consists of two \emph{sorts}; the first is $G$ and the second is $X$. 
In the \emph{action language}, we use the same logical symbols as the group language, except that different variables are reserved for different sorts; for instance, we will often use $g_0,g_1,\ldots$ for group elements and $U_1,\ldots$ for regular open sets, although we will often use even simpler notation.
The \emph{action signature} contains the group signature along with the unary binary operations $\supp^e$ and $\appl$, which will respectively mean
$$\supp^e g:=\inte \cl (M\setminus \fix g), \quad \appl(g,U)=g(U).$$
Again for brevity, we prefer to use the notation $g(U)$ instead of $\appl(g,U)$. Additionally, we includes the Boolean operation symbols such as
$$U\cap V, U\oplus V:=\inte\cl (U\cup V), U^\perp:=\inte\cl (X\setminus U), \varnothing, \mathbb{X}:=X.$$

The action structure $\ACT(M,G)$ is suitable to deal with many crucial topological properties of manifolds. Let us begin with the following basic observations. The proofs are given in~\cite{kkdln2025}, and the difference of the categories ($C^0$ vs $C^\infty$) causes no subtleties here.

\begin{lemma}\label{lem:top-prop}
There exist action formulae which express the following
for all $(M,G)\in\MMv$.
    \begin{enumerate}
\item\label{p:conn}
Connectedness of a regular open set $U$.
\item\label{p:cc}
    	That a regular open set $U$ is a connected component of a regular open set
     $V$, i.e.~$U\in\pi_0(V)$.
\item\label{p:ucc}
    	That a regular open set $U$ is a union of connected components of a
     regular open set $V$.
\item\label{p:cck}
For all natural numbers $k$, that a regular open set $U$ has exactly $k$ components, i.e $\#\pi_0(U)=k$.
\item\label{p:disj} That regular open sets $U$ and $V$ satisfy $U\oplus V=U\sqcup V$;
here $\oplus$ denotes the Boolean join and $\sqcup$ denotes disjoint union.
\item\label{p:ccpartition}
That a regular open set $W$ is a disjoint union of regular open set $U$ and $V$, each
of which consists of connected components of $W$.
\item\label{p:cnteq}
That whenever at least one of two regular open sets $U$ and $V$ has finitely many components, $U$ and $V$
have the same number of components.
\end{enumerate}
\end{lemma}

The action structure $\ACT(M,G)$ is seemingly more powerful than the group structure $G$. However, under certain hypotheses these two structures turn out to be ``equally rich'' in the following sense.
Interested readers can refer to \cite{kkdln2025, Rubin1989} for the general definition of \emph{uniform interpretability}, although we will explain its meaning in the current circumstance. 

\begin{thm}[{Rubin's Expressibility Theorem, {\cite[Theorem 2.5]{Rubin1996}}}]\label{thm:rubin}
Let $\CLM'\subseteq\CLM$ be a subclass such that for each $(X,G)\in\CLM'$ and for each $U\in\RO(X)$, there exists some $g\in G$ such that $U=\suppe g$.
Then the action structure $\ACT(X,G)$ is interpretable in $G$ uniformly for $(X,G)\in\CLM'$.
This means, we have a group formula $\phi_{\mathrm{eq}}$ and a map $\varphi\mapsto \varphi^*$ from the action language to the group language such that the following hold for each $(X,G)\in\CLM'$.
\begin{itemize}
    \item 
    For all $g,h\in G$ we have
    $$
     \suppe g = \suppe h \Longleftrightarrow \phi_{\mathrm{eq}}(g,h).
    $$
    \item For each action formula $\varphi(g_1,g_2,\ldots,U_1,U_2,\ldots)$, the set 
    $$
    \{
    (g_1,g_2,\ldots,h_1,h_2,\ldots)\mid
    \varphi(g_1,g_2,\ldots,\suppe h_1,\suppe h_2,\ldots)\}
    $$ coincides with the set of tuples satisfying the group formula $\varphi^*$.
    \end{itemize}
\end{thm}
An explicit form of the formula $\phi_{\mathrm{eq}}$ can be found in~\cite[Theorem 3.4.3]{KK2021book}. 
As a consequence, we have the following; see also~\cite[Lemma 2.11]{kkdln2025}.
\begin{cor}[{\cite[Corollary 1.19]{Rubin1996}}]\label{c:l2act}
Whenever $(M_i,G_i)\in \MMv$ for $i=1,2$ satisfy $G_1\equiv G_2$,  we have that $\ACT(M_1,G_1)\equiv\ACT(M_2,G_2)$
\end{cor}
\begin{proof}
In view of Theorem~\ref{thm:rubin}, it only remains to prove for each $(M,G)\in\MMv$ and for each regular open set $U$
we can find $g\in \Diff^\infty_{0,\vol}(M)$ such that $\suppe g = U$.
The argument is similar to that of~\cite[Propsition 2.7]{kkdln2025}, with the difference being $g$ is required to be smooth.
That is, writing a dense open subset of $U$ as a countable union of disjoint balls, one produces elements of $G$ supported
on each of these balls separately. To guarantee that their composition is an element of
$G$, one need only arrange that the derivatives on each component decay to the identity
sufficiently quickly, which can be achieved through standard techniques. 

We conclude that $U$ coincides with the extended support of the resulting diffeomorphism.
\end{proof}

As an immediate consequence, as well as a showcase of the expressive power of an action structure, let us now establish that the action structure (and hence, the group structure) detects the connectivity of the ambient manifold.
\begin{prop}\label{prop:cnt}
There exists an action sentence $\phi=\phi^{(\vol)}_{\mathrm{cnt}}$ such that for all $(M,G)\in\MMv$, the manifold $M$ is connected iff $G\models\phi$. 
\end{prop}
\begin{proof}
We let a sentence $\phi$ express the following: 
for all nonempty disjoint $U,V\in\mathrm{RO}(M)$, 
there exist nonempty disjoint regular open sets $U',U''$ in $U$ and $g\in G$ such that $g(U')\subseteq V$ and $g(U'')=U''$. It is clear that $\phi$ holds if and only
if $M$ is connected; indeed, if $N$ is connected then we can apply
Lemma~\ref{lem:cr-trans-balls-gen}, and if $M$ is not connected then we choose $U$ and $V$ to be distinct connected components.
\end{proof}

Although we do not have an expression that $\vol(U)\le \vol(V)$, we do have a relevant, slightly more technical expression.
\begin{lemma}\label{lem:vol-tech}
There exists an action formula $\vol_{\leq}(u_1,u_2,v)$ such that 
for all $(M,G)\in\MM_{\vol}$ and for all  $(U_1,U_2,V)$ of regular open sets with $U_1$ and $U_2$ connected  and $U_1,U_2\subseteq V$ with $\cl V\sse \inte M$, we have the following:
    \begin{enumerate}
    	\item
    	If $\omega(U_1)\leq \omega(U_2)$ then $\vol_{\leq}(U_1,U_2,V)$ holds.
    	\item
    	If $\vol_{\leq}(U_1,U_2,V)$ holds then
     $\omega(U_1)\leq\omega((\cl U_2)\cap V)$.
    \end{enumerate}
\end{lemma}

Lemma~\ref{lem:vol-tech} has a continuous analogue in~\cite[Lemma 3.9]{kkdln2025},
and the proof is identical after employing Lemma~\ref{lem:trans-measure} and Lemma~\ref{lem:cr-trans-balls-gen} for the current (smooth) case. This yields one of the most crucial expressions of being ``compactly contained in a ball''. The logic of the argument is identical, employing the previous lemma:
\begin{lemma}[cf.~\cite{kkdln2025}, Lemma 3.10]\label{lem:compact-contained}
 For $(M,G)\in\MMv$ and for $U,V\in\RO(M)$, 
 there exists an action formula $\phi_{\mathrm{cc}}^{(\vol)}(U,V)$ expressing that $U$ is compactly contained in an open ball that is contained in $V$;
that is, there exists a smooth collared ball $B\subseteq V$ such that $U\subseteq\inte B$.
\end{lemma}

We can now detect a manifold being closed by an infinite disjuction of sentences.

\begin{prop}\label{prop:closed}
There exists an action sentence $\phi_n=\phi^{(\vol)}_{\mathrm{closed},n}$ for $n\in\N$ such that for all $(M,G)\in\MMv$, the manifold $M$ is closed iff $G\models\phi_n$ for some $n$. 
\end{prop}
\begin{proof}
Again, it suffices to find an $\ACT$ sentence expressing the compactness.
We let a sentence $\phi$ express that the ambient manifold coincides with
\[
\bigoplus_{i=1}^n U_i\]
for some finite list $\{U_1,\ldots,U_n\}$ of regular open sets such that each $U_i$ is compactly contained in some open collared ball of the manifold; this is possible by Lemma~\ref{lem:compact-contained}. 
Then for $(M,G)\in\MMv$, the manifold $M$ is closed if and only if $H\models \phi_n$ for some $n\in\omega$. 
\end{proof}
Note that, \emph{a priori}, such an $n$ will depend on $M$; we will later remove this restriction and make the sentence universal by interpreting arithmetic, under certain hypotheses (Proposition~\ref{prop:compact-noncompact}).

\begin{prop}\label{prop:dim1}
There is an action sentence expressing that $M$ has dimension at least two for each $(M,G)\in\MMv$.
\end{prop}
\begin{proof}
We simply let a sentence express that for some nonempty regular open sets $U_i\subseteq V_i$ for $i=1,2,3$ such that
$\phi_{\mathrm{cc}}^{(\vol)}(U_i,V_i)$ and
such that $\{V_1,V_2,V_3\}$ is a pairwise disjoint collection, the set $(U_1\oplus U_2\oplus U_3)^\perp$ should be connected.
\end{proof}

\subsection{On AGAPE structures}\label{ss:agape}
We can now extend an action structure to a larger structure, called AGAPE.
The process of extension is based on the ideas and terminology of \cite{kkdln2025}, but we tailor the definition for the proof of Main Theorem. 

\emph{We mostly consider the class $\MMc$ of all pairs $(M,G)\in\MMv$ such that $M$ is closed and connected, unless specified otherwise.} 
In view of Main Theorem, our restriction to $\MMc$ will be sufficient by Propositions~\ref{prop:cnt}, ~\ref{prop:closed} and~\ref{prop:dim1}.

The universe of the new structure $\AGAPE(M,G)$ for $(M,G)\in\MMc$ is the disjoint union of the following sets:
$$
G, \RO(M), M, M^{\mathrm{fin}}, \mathbb{R},2^{\mathbb{N}},\{C(\mathbb{R}^k,\mathbb{R}^\ell)\}_{k,\ell\in\omega}.
$$
Here, we let $M^{\mathrm{fin}}$ denote the set of all finite subsets of $M$.
We also include $\mathbb{N}$ as a distinguished (definable by a formula) subset of $\mathbb{R}$.
The signature of $\AGAPE$ consists of the following symbols.
\begin{enumerate}[(i)]
\item[(AGAPE1)] The signature of $\ACT$.
\item[(AGAPE2)] The signature of the second order arithmetic for $2^{\mathbb{N}}$, $\mathbb{R}$ and $C(\mathbb{R}^k,\mathbb{R}^\ell)$:
$$
(0,1,<,+,\times,\in,\subseteq,\|\cdot\|).
$$
\item[(AGAPE3)] The signature dealing with the group actions on the points and subsets of $M$:
$$
(\appl,\#\pi_0, \in,,\cup, \in\on{cl},\exp)
$$
\end{enumerate}
The meanings of the new nonlogical symbols are the standard ones. For instance, the symbol $\|f\|$ means the uniform norm of a continuous map $f\in C(\mathbb{R}^k,\mathbb{R}^\ell)$.
We have that $p\in\cl U$ iff $p$ does not belong to $U^\perp$.
We define $\appl(g,X)=g(X)$ when $g\in G$
and $X$ is either a point or a subset of $M$. 
The expression $k=\#\pi_0(X)$ means $X$ has exactly $k$ connected components.
The set theoretic symbols in [AGAPE2] denote membership relations in $\mathbb{N},\mathbb{R}$ and $2^{\mathbb{N}}$. 
Such symbols are also used for points and subsets of $M$ in [AGAPE2], and this will not cause confusion for us.
We let $\exp(g,k,p)=q$ for $g\in G$ and $k\in\mathbb{N}$ iff $g^k(p)=q$.

The AGAPE structure will be expressive enough for us, leading to detection of manifolds and regularities. In this subsection, we establish that the extended structure  $\AGAPE(M,G)$ is still uniformly interpretable to $\ACT(M,G)$, and hence to $G$.
Our proof is based on the idea of \cite{kkdln2025}, only sketching the steps that are completely identical. We emphasize again that in this latter paper, each group is assumed to contain $\Homeo_{0(,\vol)}(M)$ although in the current case we only have $\Diff_{0(,\vol)}^\infty$, and hence, special considerations need to be given for existential statements.

\subsubsection{Interpreting second order arithmetic}
The interpretation of second order arithmetic can be carried out in a similar fashion to~\cite[Section 4]{kkdln2025}.
Namely, we will interpret $\bN$ as the collection of regular open sets with finitely many connected components.
The first step of such an interpretation would be detecting whether or not a given regular open set has finitely many connected components.
In~\cite{kkdln2025}, we found a formula which essentially asserted that a regular set $W$ has infinitely many components if there is a subset $U$ of $W$ with a component-wise injection between components of $W$ and $U$, and a homeomorphism which permutes components of $W$ in a translation-like fashion. Such a formula is not available in the smooth case, since one may not be able to find a diffeomorphism with the desired properties. Some modifications are thus necessary.

\begin{lem}[{Compare with~\cite[Section 3.4.2]{kkdln2025}}]\label{lem:infcomp}
There exists an action formula $\on{infcomp}(w)$ such that whenever $(M,G)\in\MMc$ we have $\on{infcomp}(W)$ for $W\in \RO(M)$ iff $W$ has infinitely many connected components.
\end{lem}

To explain the intuition behind Lemma~\ref{lem:infcomp}, we pass to a subset
$U\subseteq W$ which meets every component of $W$ in at most one component; moreover,
there will be another regular open $V$ ``interlaced" with $U$. One can imagine that
$U$ and $V$ are linearly arranged, with alternating components
\[\{U_0,V_0,U_1,V_1,U_2,V_2,\ldots\}.\]
We then insist that there be an element of the ambient
group $G$ which sends some point of $U_i$ into a unique component $V_i$ of $V$, and sends
a point of $V_i$ into a unique component $U_{i+1}$ of $U$. This argument is a slightly
modified version of an argument given in~\cite{kkdln2025}, with extra care given to
the existence of a diffeomorphism with the desired properties (and not just a
homeomorphism). Note that there is a formula which expresses that a regular open set $U$
is \emph{dispersed}, i.e.~every connected component $\hat U$ of $U$ is compactly contained in an open
ball inside of $U^{\perp}\oplus \hat U$.

\begin{proof}[Proof of Lemma~\ref{lem:infcomp}]
    The formula $\on{infcomp}(w)$ will express the conjunction of the following
    clearly expressible statements:
    
    There exist regular open sets $U,V\in\RO(M)$, a unique component
    $U'\in\pi_0(U)$, and an element $g\in G$ such
    that:
    \begin{enumerate}
        \item $\varnothing\neq U\subseteq W$.
        \item Every component of $W$ contains at most one component of $U$.
        \item $U$ and $V$ are dispersed and disjoint.
        \item For all components $\hat U$ of $U$, there is a unique component
        $\hat V$ of $V$ such that $g(\hat U)\cap \hat V\neq \varnothing$,
        and for all components $\hat U\in\pi_0(U)\setminus U'$, there is a
        unique component $\hat V$ of $V$ such that $g^{-1}(\hat U)\cap\hat V\neq
        \varnothing$.
        \item For all components $\hat V$ of $V$, there is a unique component
        $\hat U\in\pi_0(U)\setminus U'$ such that
        $g(\hat V)\cap\hat U\neq\varnothing$, and for all components $\hat V$ of
        $V$ there is a unique component $\hat U$ of $U$ such that
        $g^{-1}(\hat V)\cap\hat U\neq\varnothing$.
        \item $g^{-1}(U')\cap V=\varnothing$.
    \end{enumerate}

    For the forward direction, suppose that $W$ satisfies $\on{infcomp}(W)$.
    It clearly suffices to prove that $U$ has infinitely many components.
    We let $U'=U_0$. We set $V_0$ to be the unique component of $V$ such that
    $g(U')\cap V_0\neq\varnothing$.

    Having defined $U_i$ and $V_i$, assume by induction that \[\{U_0,\ldots,U_i\}
    \quad
    \textrm{and}\quad \{V_0,\ldots,V_i\}\] are all distinct.
    We set $U_{i+1}$ to be the unique component of
    $U$ which satisfies \[g(V_i)\cap U_{i+1}\neq\varnothing.\] Note that
    \[U_{i+1}\notin\{U_0,\ldots,U_i\}.\] Indeed, suppose $U_{i+1}=U_j$ for some
    $j\leq i$. It is immediate that $j\neq 0$ since $g^{-1}(U_0)\cap V=\varnothing$.
    So, we may conclude that $g^{-1}(U_j)$ meets two distinct components of $V$,
    a contradiction. Similarly, we set $V_{i+1}$ to be the unique component of
    $V$ such that $g(U_{i+1})\cap V_{i+1}\neq\varnothing$. It is immediate that
    $V_{i+1}\notin \{V_0,\ldots,V_i\}$ since otherwise $g^{-1}(V_{i+1})$ would
    meet two distinct components of $U$, a contradiction. It follows that $U$,
    and hence $W$, both have infinitely many components.

    Conversely, suppose that $W$ has infinitely many components. 
    By the compactness of $M$,
    we may choose
    a convergent sequence of points $\{x_i\}_{i\in\N}\subseteq W$,
    each term lying in a different component of $W$. Passing to a further
    subsequence if necessary, we may assume that this sequence lies in a single
    Euclidean chart diffeomorphic to $\R^n$ (or to the upper half-space, in the case
    of a manifold with boundary), and that the sequence converges to
    the origin $0$. Passing to an even further subsequence if necessary,
    we may assume
    $d(x_{i+1},0)\leq \frac{1}{2}d(x_{i},0)$ for all $i$. We set $U$ to be a union
    of neighborhoods $\{U_i\}_{i\in\N}$ of the points
    $\{x_i\}_{i\in\N}$, where we take the radius
    of the neighborhood $U_i$ about $x_i$ to be $\frac{1}{4}d(x_i,0)$. Choose
    smooth paths $\{\gamma_i\}_{i\in\N}$ connecting $U_i$ to $U_{i+1}$ such that
    for all $y\in\gamma_i$, we have \[\frac{5}{4}d(x_{i+1},0)\leq d(y,0)\leq
    \frac{3}{4}d(x_i,0),\] and let $\tilde V$ be a small tubular neighborhood of
    the union of these paths. We shrink $\tilde V$ slightly to obtain the set $V$.
    Since we have freedom in shrinking $V$, it is clear that we may arrange
    the existence of a $C^{\infty}$ diffeomorphism $g$
    with the desired properties; moreover,
    we may arrange for $g$ to preserve the standard volume form on $\R^n$ in the
    volume-preserving case, provided $n\geq 2$.
\end{proof}

Evidently, there is formula $\on{fincomp}(w)$ which is simply the negation of
$\on{infcomp}(w)$, and expresses that a regular open set has only finitely many
components. Now, the interpretation of the remaining sorts in the second order arithmetic is precisely done in the manner of 
\cite[Section 4]{kkdln2025}.

We thus obtain a uniform interpretation of first order analysis within the first order
theory of $G$. 
\begin{prop}\label{prop:arith}
The second order arithmetic is interpretable in $\ACT(M,G)$ uniformly over $(M,G)\in\MMc$.
\end{prop}

\subsubsection{Points, exponentiation, and covering}
It remains to interpret the [AGAPE3] sorts as well as relevant operations described at the beginning of this subsection.
The interpretation of finite tuples of points is similar to~\cite[Section 4]{kkdln2025}, though 
the argument here is simpler as we only consider $(M,G)\in\MMc$, namely $M$ is closed.
Specifically, a point $p\in M$ is coded by a regular open set $U$ as follows:
\begin{enumerate}
    \item $U$ has infinitely many components.
    \item For all components $\hat U$ of $U$, the complement $\hat U^{\perp}$ consists of
    exactly two components, say $V_1^{\hat U}$ and $V_2^{\hat U}$;
    exactly one of these components, say $V_1^{\hat U}$, contains finitely many components
    of $U$, and $V_2^{\hat U}$ contains infinitely many components of $U$. Moreover, 
    $V_2^{\hat U}$ is compactly
    contained in $V_2^{\hat U}\oplus \hat U$.
    \item For all pairs $W_1,W_2\in\RO(M)$ of regular open sets such that $W_i$ is compactly
    contained in $W_j^{\perp}$ for $i\neq j$, we have all but finitely many components of $U$
    are contained in $W_i^{\perp}$ for at least one $i\in\{1,2\}$.
\end{enumerate}

The preceding conditions are clearly first order expressible. Notice furthermore that
an easy compactness argument shows that
\[\bigcap_{\hat U\in\pi_0(U)}V_2^{\hat U}\neq\varnothing.\] The separating condition
involving
the sets $W_1$ and $W_2$ shows that this intersection actually consists of exactly one point
$p$, the point coded by $U$.

We declare two sets $U_1$ and $U_2$ to be equivalent if for all $\hat U_1\in\pi_0(U_1)$,
infinitely many components of $U_2$ are contained in $V_2^{\hat U_1}$. Clearly this condition
will hold if and only if the points coded by $U_1$ and $U_2$ are equal, and it is
straightforward to see that this condition is actually symmetric and defines a (first order
definable) equivalence relation. It is straightforward also to generalize this condition to
code finite tuples of points in $M$ by regular open sets, and to interpret the exponentiation function $\exp$.
\begin{prop}\label{prop:points-exp}
The following are interpretable in $\ACT(M,G)$ uniformly over $(M,G)\in\MMc$:
\begin{enumerate}
\item The set $M^{<\omega}$ of finite tuples of points in $M$;
\item The exponentiation relation $g^k(p)=q$ for $p,q\in M$ and $k\in \bN$.
\end{enumerate}
\end{prop}

From the interpretation of points, it is straightforward to define a predicate expressing that
a point lies in the closure of a regular open set $U$, since $p\in \cl U$ iff $p\not\in U^\perp$.
We now conclude:
\begin{prop}\label{p:agape-ext}
Uniformly for $(M,G)\in\MMc$, the structure $\AGAPE(M,G)$ is interpretable to $\ACT(M,G)$.
\end{prop}

Henceforth, formulae and sentences will implicitly be in the signature of the structure AGAPE as discussed at the beginning of this section. We emphasize that since this is a conservative expansion of the language of group theory, all formulae and sentences can be reduced to the language of group theory.

\subsection{On (certain definable) collared topological balls}\label{ss:balls}
One of the most crucial logical steps for Main Theorem is an interpretation of certain collared balls, which are abound enough to generate the topology of $M$. To begin with, let us note it is expressible in $\AGAPE$ language that some finite collection
$\{U_1,\ldots,U_n\}$ of regular open sets cover the closure of another regular open set $W$.
We can interpret enough about Lebesgue covering dimension to prove the following:

\begin{prop}\label{prop:dimension}
For all $n\in\N$ there is a sentence $\dim_n$ such that $G\models\dim_n$ if and only if the underlying manifold $M$ has dimension equal to $n$.
\end{prop}
The proof of Proposition~\ref{prop:dimension} is carried out identically to Theorem 6.1 in~\cite{kkdln2025}.
Similarly, Section 6.2 of~\cite{kkdln2025} carries through without change for diffeomorphisms,
and so we obtain:
\begin{lem}\label{lem:balls}
There exist $\AGAPE$ sentences $\flows_n$ and $\param_n$ such that
the following hold for all $(M,G)\in\MMc$ with $\dim M=n$:
\begin{enumerate}
\item We have a map $\Phi$ from 
$$\DD_0:=\{(U,\underline{g},p)\in\RO(M)\times G^n\times M : \flows_n(U,\underline{g},p)\}$$
to the set 
$$
\{f: [-1,1]^n \to M |\; f\text{ is a collared topological ball in }M\}
$$
such that whenever $(U,\underline{g},p,\underline{r},q)\in\DD_0\times B(1)\times M$ 
we have
$$ \Phi(U,\underline{g},p)(\underline{r})=q \Longleftrightarrow \param_n(U,\underline{g},p,\underline{r},q),$$
and such that $\models\param_n(U,\underline{g},p,\underline{0},p)$.
\item\label{p:flows2}The collection of all smooth collared balls $U$ satisfying
\[\models(\exists \underline{\gamma}\exists \pi)\,\on{flows}_{n}(U,\underline{\gamma},\pi)\]
forms a basis for the topology of $M$.
\end{enumerate}
\end{lem}
Recall that in a structure, a \emph{definable set} is a subset of a certain tuples of the universe consisting of the elements satisfying some first order formula. Such a set is  \emph{uniformly definable} if the formula is the same for all choices of a structure in a given fixed collection. Uniformly definable relations or functions are defined in the same way.
In words, Lemma~\ref{lem:balls} says that we have a uniformly definable collection of collared topological balls in each dimension of $\MMc$ that are fine enough to generate the topology of a smooth manifold.
The reader will note that we do not assert that all (smooth or topological) collared balls are definable; the image of $G$ does not act transitively on such balls.

Using the preceding discussion, we will detect homeomorphism types, and also compactness by first order sentences.

\begin{thm}\label{thm:homeomorphism type}
Let $(M,G),(N,H)\in\MMv$ such that $M$ is closed and connected.
Then there exists a sentence $\phi_{M,0}\in \Th{G}$ such that for all smooth (not necessarily closed
    or even compact) manifolds $N$ and
    \[\Diff^\infty_{0(,\vol)}(N)\leq H\leq\Homeo_{(\vol)}(N),\]
    we have $\phi_{M,0}\in\Th{H}$ if and only if $M$ and $N$ are
    homeomorphic.
\end{thm}
\begin{proof}
First we can distinguish the one-dimensional and higher dimensional cases by Proposition~\ref{prop:dim1}.
We also let $\phi_{M,0}$ to force that the ambient manifold is connected, and covered by a certain number of collared balls (hence closed). 
The conclusion is now clear when $\dim M=1$ by Proposition~\ref{prop:dim1}. 
For $M$ of dimension two or more, the hypotheses of Lemmas 8.1 and 8.2 in~\cite{kkdln2025} are satisfied, in which case the conclusion of Theorem 1.4 in that paper holds.
\end{proof}

In particular, the ``$M$ and $N$ are homeomorphic'' part of the conclusion in Theorem~\ref{thm:main} is now proved.
\subsection{Detecting compactness}
The content of this subsection is not relevant to the proof of Main Theorem, but is interesting in that the first order theory of groups can certify the compactness of a manifold.
Regarding compactness, we need to impose the hypothesis that $\partial M=\varnothing$; indeed, there is no way to distinguish between $\R$ and $I$ from the first order theory of their respective homeomorphism groups because the groups are in fact isomorphic to each other.

\begin{prop}\label{prop:compact-closed}
There exists a group sentence $\psi$ such that whenever $(M,G)$ belongs to $\MMv$ with $M$ being compact,
we have $\models \psi$ iff $M$ is closed.
\end{prop}
\begin{proof}
Let $(M,G)$ be as given. Note that Lemma~\ref{lem:infcomp} and its proof hold in verbatim even if we assume $M$ is compact rather than closed. Hence as in Lemma 3.14 of \cite{kkdln2025}, we have an action formula $\on{finint}(U,V)$ expressing that only finitely many components of $U$ intersects $V$. Using this we can define an action formula $\on{touch}_\partial(U)$ expressing that the closure of $U$ nontrivially intersects $\partial M$, also as in~\cite{kkdln2025}. This certifies that $M$ is not closed by a group sentence, by the interpretability (Theorem~\ref{thm:rubin}).
\end{proof}

\begin{prop}\label{prop:compact-noncompact}
    There exists a group sentence $\psi$ which is true in the homeomorphism group of an arbitrary connected boundaryless manifold $N$ if and only if $N$ is compact.
\end{prop}
\begin{proof}
    We have that $N$ is compact if and only if it is sequentially compact.
    Thus, a boundaryless manifold $N$ is noncompact if and only if
there is a regular open set $U\in\RO(N)$ such that:
\begin{enumerate}
    \item $U$ has at least $n$ components for every $n\in\N$.
    \item For all $W\subseteq U$ with the property
that each component of $U$ 
contains exactly one component of $W$, and
for all $V\in\RO(N)$ with $V$ compactly
contained in a collared open ball in $N$, there exists a subset $W_0\subseteq W$ such that:
\begin{enumerate}
    \item Each component of $W$ contains exactly one component of $W_0$.
    \item $W_0\cap V$ has finitely many components.
\end{enumerate}
\end{enumerate}
In order to express part (1) above, we use the interpretation of arithmetic by regular open sets contained in some compact balls, and then we express counting of connected components for such open sets. This way, we can quantify over natural numbers.
Part (2) is clearly expressible.

To see that these conditions give the desired sentence $\psi$,
suppose that $N$ is compact, and let
$U$ be a regular open set with infinitely many components.

Fix a compatible metric on
$N$ and choose a dispersed sequence
collared balls, one in each of the components of $U$, with
shrinking radii; call their union, which is equal to their join, $W$. Note that these balls accumulate on some point $p\in N$, since $N$ is
compact. Choose an arbitrary $V$ containing $p$ which is compactly contained in a ball in
$M$, we observe that for any suitable $W_0$ as given by the conditions above,
the set $V$ must contain infinitely many
components of $W_0$.

Conversely, suppose at $N$ is not compact, and choose an exhaustion
\[N=\bigcup_{i\in\omega} K_i\] by increasing compact sets; for the sake of concreteness,
we may choose a compatible metric on $N$ and let $K_i$ consists of the closure of a ball
of radius $i$ centered at fixed but arbitrary point.
Choose $U$ such that for all $i$, at least
one component of $U$ lies in $N\setminus K_i$ and $U\cap K_i$ is a regular open set
with finitely many components. Choosing $W\subseteq U$ arbitrary, we see that
components of $W$ exit every compact subset of $M$. If $V$ is compactly contained in a ball
in $N$ then $V$ is contained in some $K_i$. Choosing sufficiently small balls in each
component of $W\cap K_i$, we arrange for $W_0$ to meet $V$ in finitely many components.
\end{proof}

With some small modifications, one may deal with the case of disconnected manifolds;
we omit the details for brevity.
An analogous result holds for measure preserving homeomorphisms in dimension at least two,
and for diffeomorphisms. In particular:

\begin{thm}\label{thm:cpt-detect}
There is a sentence $\psi_{\on{cpt}}$ such that whenever a pair $(M,G)$ satisfies one of the following conditions 
we have $G\models \psi_{\on{cpt}}$ iff $M$ is compact:
\begin{enumerate}
    \item $M$ is a topological manifold with $\partial M=\varnothing$
    and 
    \[
    \Homeo_{0(,\vol)}(M)\le G\le \Homeo_{(\vol)}(M).\]
    \item $M$ is a smooth manifold with $\partial M=\varnothing$
    and 
    \[
    \Diff^\infty_{0(,\vol)}(M)\le G\le \Diff^1_{(\vol)}(M).\]
\end{enumerate}
\end{thm}

\begin{rem}
    We note that if $(M,G)$ is as in Theorem~\ref{thm:cpt-detect} then for all $n$ one
    can express that $M$ has at least $n$ noncompact ends. Indeed, one expresses the
    existence of a regular open set $U$ such that $U^{\perp}$ contains at least $n$
    components $\{V_1,\ldots,V_n\}$ which are themselves noncompact; the noncompactness
    of each $V_i$ is checked as in the proof of Proposition~\ref{prop:compact-noncompact}. In particular, the theory of $G$ determines the number of noncompact
    ends of $M$.
\end{rem}

\section{Manifold Recognition}\label{s:mfd}
We now gather some generalities from differential topology which will allow us to distinguish between different smooth structures from the first order theory of diffeomorphism groups. 
\emph{For brevity, we restrict our attention to the subclass $\MM'_{(\vol)}$ of $\MMc$, consisting of pairs $(M,G)$ such that $M$ is a smooth, closed, connected manifold  and such that $$\Diff^\infty_{0(,\vol)}(M)\le G\le\Diff^1_{(\vol)}(M).$$.}

\subsection{Defining smooth atlases}
From Lemma~\ref{lem:balls}, we are able to define a topological atlas on $M$,
which is to say a covering of $M$ by topological collared balls; specifically, we can
express the fact that $\mathcal U=\{U_1,\ldots,U_n\}$ covers $M$ and that each $U_i$ is a
topological collared ball. In order to prove first order rigidity of diffeomorphism groups
of $M$, we will need to be able to discuss a smooth atlas, i.e.~when the elements of 
$\mathcal U$ are smooth collared balls in $M$. Recall here that a \emph{smooth collared
ball} in $M$ (which has dimension $n$) is a map \[\phi\colon B^n(1)\longrightarrow M\] which
is a diffeomorphism onto its image, which extends to a map $\Phi$ defined on $B^n(2)$,
which is also a diffeomorphism onto its image.

The main technical fact which allows us to achieve this goal is the Bochner--Montgomery theorem, which we will apply for the case $G:=(0,1)^n$ equipped with the addition operation.

\begin{thm}[{\cite[Theorem 4]{BM1945}}]\label{thm:bochner-montgomery}
Let $X$ be a smooth, connected, boundaryless manifold, and let $k\in\bN_+$.
If we have an action
$$\Phi: G\times X \to X$$
of a local Lie group $G$ such that $\Phi$ is jointly continuous,
and such that $x\mapsto\Phi(g,x)$ is $C^k$ for each $g\in G$, then $\Phi$ is jointly $C^k$.
\end{thm}

Recall from Lemma~\ref{lem:balls} that we have an AGAPE-definable set  $$\DD_0\subseteq \RO(M)\times G^n\times M$$ uniformly over $(M,G)\in\MM^c_{(\vol)}$ with $\dim M=n$; furthermore,  each $\delta\in\DD_0$ determines a topological ball $\phi_\delta:B(1)\to M$ such that $\phi_\delta(r)$ is the unique $q\in M$ with
$$\AGAPE(M,G)\models \param_n(\delta,r,q).$$ We will impose further first-order conditions on $\DD_0$ 
such that Theorem~\ref{thm:bochner-montgomery} forces parametrized topological collared balls to be smooth.
For this, it will be useful for us to explain more details on the construction of $\phi_\delta$ in~\cite{kkdln2025}.

\begin{prop}\label{prop:smooth-ball}
There exists an $\AGAPE$--definable subset $\DD_1$ of $\DD_0$ uniformly over $(M,G)\in\MM'_{(\vol)}$ with $\dim M=n$ such that the following hold:
\begin{enumerate}
    \item For each $\delta\in\DD_1$, the map 
    $$\phi_\delta:[-1,1]^n\to M$$ defined above is a $C^1$--collared ball.
    \item The collection of the images of $\phi_\delta$ with $\delta\in\DD_1$ forms a basis for $M$.
\end{enumerate}
\end{prop}

\begin{proof}
We will apply Bochner--Montgomery Theorem to force that a definable parametrizable ball is $C^1$--smooth. We take extra
care so that the theorem is applied  on a smooth submanifold of $M$ or of $\mathbb{R}^n$, as is required in the hypothesis of that theorem. Let us set $Q:=(-1,1)^n$. Note we have a local translation action of the local Lie group $Q$ on $Q$. 

The proof is divided into three steps. In the first step, we prove that a parametrized topological ball $\phi$ can be forced to be a $C^1$ map. In the second, we force $\phi^{-1}$ is also $C^1$. In the third, we provide a model for such $\phi$ so that the images can be made to be an arbitrarily small neighborhood of a given point.

For the first step, let us set $\phi:=\phi_\delta: \cl Q \to M$ defined by a fixed parameter $\delta\in\DD_0$ in the $\AGAPE$ language. We will establish that $\phi$ is $C^1$ near an arbitrary point $v_0\in Q$.
We pick a smooth chart $u:Q\to M$ such that $y_0=\phi(v_0)=u(0)$
and such that $\cl u(Q)\subseteq \phi(Q)$. 
Let us define $H$ to be the local action of $Q$ on the smooth manifold $u(Q)$ defined as follows:
$$
(v,y)\stackrel{H}{\mapsto} \phi(v+\phi^{-1}(y)).
$$
Note that for all $y\in u(Q)$, the action is defined on the zero neighborhood 
$$
v\in  \phi^{-1}u(Q)-\phi^{-1}(y).
$$
Using the AGAPE language, we can further require that 
whenever $v\in Q$ there exists some $g_v\in G\le\Diff^1(M)$ such that 
$$\phi(v+\phi^{-1}(y))=g_v(y)$$
for all $y\in \phi\left(Q\cap (Q-v)\right)$.
By Bochner--Montgomery Theorem, the map
$$
(v,y_0)\to H(v,y_0)=\phi(v+v_0)
$$
is $C^1$ for $v\in \phi^{-1}u(Q)-v_0$. It follows that $\phi$ is $C^1$ near $v_0$.

As the second step, let us establish that $\phi^{-1}$ is $C^1$ near $y_0=\phi(v_0)$.
We let $Q_1:=c_1Q$ for some sufficiently small, definable $c_1>0$ such that $0\in Q_1$. 
We can further require that for all $v\in Q_1$
there exists $h_v\in \Diff^\infty_{\vol}(\mathbb{R}^n)$
such that $\suppe h_v\subseteq 0.9 Q$
and such that every $w\in Q_1$ satisfies
$$h_v(w)=v+w.$$
For an example of such $h_v$, we can choose smooth volume-preserving flow $\Phi_1,\ldots,\Phi_n$ supported in $0.9Q$ such that $\Phi_i^r(x)=x+re_i$ for all $x\in Q_1$ and $|r|\le c_1$. Here, $e_i$ denotes the $i$--th standard basis vector. Then we may set 
$$h_v:=\prod_i \Phi_i^{r_i}$$
for $v=(r_1,\ldots, r_n)\in Q_1$.
Let us now require that for all $g\in G$ with $\suppe g\subseteq \phi(Q)$, 
we have $\phi^{-1}g\phi$ is $C^1$ on $Q$.
This is indeed possible since $\phi$ is definable
and since the second order arithmetic is interpreted in $\AGAPE$.
Applying this requirement to 
the volume preserving smooth diffeomorphism of $M$ extending $u\circ h_v\circ u^{-1}$ by the identity outside $u(0.9Q)$, we see for each $v\in Q_1$ that
$$
(v,x)\stackrel{H'}{\mapsto} \phi^{-1}\circ u\left(v+u^{-1}\circ \phi(x)\right)$$
is $C^1$ for $x\in \phi^{-1}\circ u(Q_1)$.
Let us pick a smooth metric ball $Q_2$ around $v_0$ contained in $\phi^{-1}u(Q_1)$.
By applying Bochner--Montomery theorem to the action $H'$ on the smooth manifold $Q_2$, we see that the map
$$v\mapsto H'(v,v_0)=\phi^{-1}\circ u(v)$$
is $C^1$, which means $\phi^{-1}$ is $C^1$ near $u(0)=y_0$.

For the last step, we need to establish $\DD_1$ chosen above contains sufficiently small smooth collared balls. 
For this, let us pick an arbitrary $p\in M$ and choose a smooth, Darboux cooridnate $\phi:Q\to M$ for the volume form around $p$; this means that the pull-back measure of $M$ coincides with the Lebesgue measure on $Q$, and such that $\phi(0)=p$. Furthermore, $\phi(Q)$ can be chosen to be arbitrarily small around $p$. 
In part (1), we required that each $v\in Q$, 
the local Lie group action
$$y\mapsto \phi(v+\phi^{-1}(y))$$
extends to some $g_v\in\Diff^\infty(M)$. This process can be done again using circular flows, and using the conjugation of $h_v$ by $\phi$ exactly as above.
\end{proof}

We remark that we cannot arrange for $\DD_1$ to parametrize of all $C^1$--smooth balls because the elements of $\DD_1$ ultimately arise from the action of $G$, and $G$ may not contain all $C^1$ diffeomorphisms of $M$.

\subsection{The Cheeger--Kister Lemma in the $C^\infty$ category}
\newcommand{\ES}[0]{\EE^{\infty}(n,k)}
Here, the notation
$A\hookrightarrow B$
will refer to an inclusion map from a set $A$ into a set $B$.
When the embedding has the name $q$, the arrow is decorated with a label as $ A\stackrel{q}{\hookrightarrow}B$. Throughout this section
and for the remainder of the paper, we give all spaces of $C^{\infty}$
maps the \emph{weak topology}, which is the union of the $C^r$ topologies
for finite values of $r$; see Section 2.1 of~\cite{hirsch2012differential}.

The following is well known; see~\cite{Brown1962b}, for example.
\begin{lem}\label{lem:exhaust}
If $W$ is a smooth or topological $n$--manifold, 
then there exists a sequence of compact $n$--manifolds $\{C_i\}_{i\in\omega}$ such that
\[W=\overline{\bigcup_{i\in\omega} C_i}\] and such that \[C_i\sse\inte C_{i+1}\] for all $i$.
\end{lem}

Although the authors believe that the following lemma belongs to the classical literature, we include a proof here for completeness. Note that its $C^0$ version is the main result of~\cite{EK1971}.
\begin{lem}\label{lem:C1-EK}
Let $M$ be a compact smooth $n$--manifold, possibly with boundary.
Let \[C\sse W\subseteq\cl W\sse M\setminus \partial M\] be subsets such that
$C$ is compact and $W$ is open.
If $q\co W\longrightarrow M$ is a smooth embedding sufficiently near from the identity,
then there exists a compact set $K$ with
\[C\subseteq\inte K\subseteq K\subseteq \inte (M)\] 
and $\alpha\in\Diff_0^{\infty}(M)$  such that the following diagram commutes:
\[
\begin{tikzcd}
& & M \\
C \arrow[rru,"q\restriction_C",shift left=1ex]\arrow[r,hook]& K \arrow[r,hook] & M \arrow[u,"\alpha"] & W\setminus K \arrow[l,hook']\arrow[lu,hook'],
\end{tikzcd}
\]
and such that $\alpha$ is as close to the identity as desired.
\end{lem}
\bpf
By Lemma~\ref{lem:exhaust}, we may assume without loss of generality that $C$ and $\cl W$ are submanifolds of $M$, so that $W$ is regular open.
Assuming $q$ is sufficiently close to the inclusion map, we may further require that the interior of some compact submanifold $K$ of $W$ contains $C\cup q(C)$. Set
\begin{align*}
    D&:=C\sqcup (M\setminus\inte K),\\
    D'&:=q(C)\sqcup (M\setminus\inte K).
\end{align*}
There is an embedding $Q\co D\longrightarrow M$,
extending $q\co C\longrightarrow q(C)$
by the identity outside $K$.
This implies that $Q$ and the inclusion map are ambiently diffeotopic; we can see this from~\cite[Proposition 4.4.4]{Wall2016book},
whose proof first uses~\cite[Corollary 2.2.5]{Wall2016book} to construct an isotopy
\[F: C\times I\longrightarrow M\] between $q_{\restriction C}$ and the inclusion $C\hookrightarrow M$ and then extends $F$ to an ambient isotopy \[G :M\times I
\longrightarrow M\] using~\cite[Theorem 2.4.2]{Wall2016book}.
In particular, $Q$ extends to a diffeomorphism $\alpha$, as required by the conclusion. 

To conclude, we need to show that the map $\alpha$ can be taken arbitrarily close to the identity in the (weak)
$C^{\infty}$ topology, provided that
$q_{\restriction C}$ is sufficiently close to the inclusion of $C$ in $M$.
That $F$ can be taken to be arbitrarily close to the constant isotopy that equals $q_{\restriction C}$ at every point in time follows from the fact that the map 
 \[H:W\times I\longrightarrow M\] used in the proof
 of~\cite[Corollary 2.2.5]{Wall2016book} is $C^{\infty}$.
 It follows that both the vector field on $C\times\bR$ with values in $T(M\times \bR)$ in the proof of~\cite[Theorem 2.4.2]{Wall2016book}, as well as its extension to
 $M\times \bR$ can be taken to be arbitrarily close in the $C^{\infty}$
 topology to the vector field $\partial_{t}$ on $M\times \bR$ that projects to $0$ and
 $\partial/\partial t$ on the two respective factors.
 The lemma now follows. 
\epf

Recall that we write $B^n(r)$ for the closed ball of radius $r$ in $\R^n$, centered at the origin. 
For positive integers $n,k$, we denote by $\EE^{\infty}(n,k)$ the set of all tuples of $C^{\infty}$--embeddings
\[
\underline f=(f_{1},\ldots,f_{k})\]
from $B^n(k+1)$ to $\bR^{2n+1}$ such that the set
\[C(\underline f):=\bigcup_{1\leq j\leq k} \im f_{j}\sse\bR^{2n+1}.\] 
is a closed connected $n$--submanifold $M$ of $\bR^{2n+1}$, with
\[M=\bigcup_{1\leq j\leq k} f_{j}(B^{n}(1)).\]

By the Whitney Embedding Theorem, every smooth, compact $n$--manifold $M$ is
$C^{\infty}$--diffeomorphic to $C(\underline{f})$ for some tuple \[\underline{f}=
\underline{f}_M=(f_{i,j})\in\EE^{\infty}(n,k)\] as above,
which we call a \emph{parametrized $C^{\infty}$--cover} of the manifold $C(\underline{f})$.
The space $\EE^{\infty}(n,k)$ inherits a topology from the weak topology
on the function space
\[C^{\infty}(B^n(k+1),\bR^{(2n+1)k}).\] This space is metrizable
since $B^n(k+1)$ is compact; see Appendix A.4 of~\cite{Wall2016book}.

We note the following rigidity result, which is a $C^{\infty}$--variation of a topological argument in~\cite{CK1970}.

\newcommand{\bde}[0]{\mathbf{\delta}}
\newcommand{\bep}[0]{\mathbf{\epsilon}}
\newcommand{\bm}[0]{\mathbf{\mu}}
\newcommand{\bn}[0]{\mathbf{\nu}}

\begin{lem}\cite{CK1970}\label{lem:c-infty-ck}
For each $\underline{f}\in\EE^{\infty}(n,k)$ and for each $\epsilon>0$,  there exists $\delta>0$ such that every $\underline{f}\in\EE^{\infty}(n,k)$ that is at most $\delta$--far from $\underline f$ by the $C^{\infty}$--distance admits a $C^{\infty}$ diffeomorphism 
\[ C(\underline{f})\longrightarrow  C(\underline{f'})\] 
that is at most $\epsilon$--far from the identity map by the $C^{\infty}$--distance. \end{lem}
\begin{proof}
	We fix $\underline{f}$ as in the statement of the lemma,
    and consider $\underline{f}'\in\EE^{\infty}(n,k)$ arbitrary. We write 
	$D=d(\underline{f},\underline{f}')$ for the $C^{\infty}$--distance.
    For all \[1\leq j\leq k\quad\textrm{and} \quad 1\leq \ell\leq k+1,\] we write
	 \begin{equation*}
	  F_{j,\ell}=f_{j}(B^{n}(\ell)),\quad  F'_{j,\ell}=f'_{j}(B^{n}(\ell)),\quad  
	 	G_{j,\ell}=\bigcup_{1\leq i\leq j}F_{i,\ell}, \quad  G'_{j,\ell}=\bigcup_{1\leq i\leq j} F_{i,\ell}.
	 \end{equation*}
One shows by induction on $1\leq j\leq k$ the existence of maps \[\bde_{j}:\bR_{>0}\longrightarrow \bR_{>0}\] with the property that for all $\epsilon>0$, the inequality \[d(\underline{f},\underline{f}')<\bde_{j}(\epsilon)\] implies the existence of a diffeomorphic embedding \[h:G_{j,\ell}\longrightarrow
    C(\underline{f}')\] with the property that 
    \[d(\iota_{\restriction G_{j,\ell}},\iota'\circ h)<\epsilon\]
for some  $C^{\infty}$-metric $d$ on $C^{\infty}(G_{j,\ell},\bR^{n})$; here, $\ell=k-j+1$, and $\iota$ and $\iota'$ denote the inclusions of $C(\underline{f})$ and $C(\underline{f}')$ into $\bR^{2n+1}$,
    respectively. Note that the conclusion of the lemma then follows by taking $\delta=\bde_{k}(\epsilon)$ for any given $\epsilon>0$, since every diffeomorphic embedding between two closed manifolds of the same dimension is a diffeomorphism.
	
    To begin, note that there are
    functions \[\bn_{j}:\bR_{>0}\longrightarrow \bR_{>0}\] for $1\leq j\leq k$, with the property that if $D\leq\nu_{j}(\epsilon)$, then there exists a diffeomorphism
    \[g_{j}:F_{j,k+1}\stackrel{\cong}\longrightarrow F'_{j,k+1}\] such that $d'(\iota'\circ g_{j},\iota_{\restriction F_{j,k+1}})<\epsilon$ for some fixed compatible metric $d'$ on 
    the space $C^{\infty}(F_{j,k+1},\bR^{2n+1})$. 
	
	The existence of $\bde_{1}$ follows from the existence of $\nu_{1}$.
    Assume now that $\bde_{j}$ is given for $1\leq j<k$, and write 
    	\[F =G_{j,k-j+1}\cap F_{j+1,k-j+1}.\]
	
	If $h$ is a diffeomorphic embedding of $G_{j,k-j+1}$ into $C(\underline{f}')$ such that 
    $\iota'\circ h$ sufficiently close to $\iota_{\restriction G_{j,k-j+1}}$ in the space $C^{\infty}(G_{j,k-j+1},\bR^{2n+1})$, then \[h(F)\subseteq F'_{j+1,k+1}.\]
    Consequently, $g_{j+1}^{-1}\circ h$ is defined on the entirety of $F$.
    Note also that the distance between $g_{j+1}^{-1}\circ h_{\restriction F}$ and the inclusion of $F$ into $C(\underline{f})$ can be made smaller than any specified
    $\lambda>0$,  provided $\iota'\circ g_{j+1}$ and $\iota\circ h$ are both at distance at most $\delta(\lambda)$ from the restriction of $\iota$ to their domain.
	
	It follows that for any given $\lambda>0$,
 an embedding $h$ as above satisfying
 \[d'(\iota'\circ h,\iota)<\lambda\] exists, provided that $D$ is sufficiently small. Moreover, for $\lambda$ small enough, Lemma
 \ref{lem:C1-EK} provides a diffeomorphism $\alpha\in\Diff^{\infty}_0(F_{j+1,k-j+1})$ which is the identity in some neighborhood of $\partial F_{j+1,k-j+1}$, and which agrees with $g_{j+1}^{-1}\circ h$
 on some neighborhood $U$ of $G_{j,k-j}\cap F_{j+1,k-j}$.
 The conclusion of Lemma~\ref{lem:C1-EK}, together with the previous paragraph,
 ensures that $\alpha$ can be taken to lie in any prescribed neighborhood of the identity
 in the $C^{\infty}$ topology, provided that $D$ is small enough.
	
Let \[h':G_{j+1,k+j}\longrightarrow C(\underline{f}')\] be given by  
	\begin{equation*}
		h'(x)= 
		\begin{cases}
		  h(x)  &  \text{if } x\in G_{j,k-j} \\
		  g_{j+1}\circ\alpha & \text{if } x\in F_{j+1,k-j}.
		\end{cases}
	\end{equation*}
	This is clearly a well-defined $C^{\infty}$-immersion between the two sets. Since $g_{j+1}\circ\alpha$ agrees with $h$ on the open set $U$, it is easy to see that this is in fact an embedding, provided that $\iota'\circ h$
    is sufficiently close to $\iota_{\restriction G_{j+1,k-j}}$; moreover,
    it can be taken to be $\epsilon$-close to the inclusion, provided $D$ is small enough.
\end{proof}

\subsection{Proof of the Theorem~\ref{thm:mfd}}
By Proposition~\ref{prop:dim1}, we may assume $\dim M\ge2$.
We construct $\phi_M$ as the conjunction of the two sentences.
The first one is $\phi_{M,0}$ in Theorem~\ref{thm:homeomorphism type},
which detects the homeomorphism type of $M$. In particular, $\phi_M$ detects the dimension $n\ge 2$ of the ambient manifold.

Let us now describe the second sentence $\phi_{M,1}$, which roughly expresses that the ambient manifold has a $C^{\infty}$--embedding
into $\mathbb{R}^{2n+1}$ that is sufficiently close to a fixed smooth--embedding of $M$ in the sense of Lemma~\ref{lem:c-infty-ck}.
For this purpose, one proceeds as in the case of homeomorphism groups in~\cite{kkdln2025}. We fix a $C^{\infty}$ embedding of $M$ into $\R^{2n+1}$, which is encoded by a tuple $\underline f\in\EE^{\infty}(n,k)$. Let us also fix a ``Cheeger--Kister constant'' $\epsilon>0$ which gives the conclusion of Lemma~\ref{lem:c-infty-ck}; namely, an arbitrary $\underline f'\in\EE^{\infty}(n,k)$ that is $\epsilon$--close to $\underline f$ in the $C^{\infty}$ topology satisfies \[C(\underline f')\cong C(\underline f)
    \cong M.\] 
By subdividing $B^n(k+1)$ into finitely many subcubes with rational corners, we may assume that the derivatives up to any fixed finite order $\ell$ of $\underline f$ vary by at most $\epsilon'\ll\epsilon$ on each subcube. Letting $\epsilon'$ be sufficiently small, we obtain a formula $\phi_{M,1}$ such that 
for a smooth closed manifold $N$ homeomorphic to $M$
and for a group $H$ with
\[\Diff^\infty_{0(,\vol)}(N)\le H\le \Homeo_{(\vol)}(N)\]
 we have $H\models\phi_{M,1}$
 if and only if all of the following hold.
\begin{itemize}
    \item 
$N$ admits a $C^1$--smooth atlas consisting of
$k$ diffeomorphic copies of $B^n(k+1)$ such that the images of $B^n(1)\subseteq B^n(k+1)$ cover $N$. 
\item 
There exists some $\underline f'\in\EE^{\infty}(n,k)$ which differs from $\underline f$ in derivatives up to degree $\ell$ by at most $\epsilon'$,
and hence, which is $\epsilon$--close to $\underline f$ in the $C^{\infty}$ topology.
\end{itemize}

We set $\phi_M$ as the conjunction of $\phi_{M,0}$ and $\phi_{M,1}$.
Part~(\ref{p:mfd1}) of the theorem is clear. 
Let $N$ and $H$ as given in part~(\ref{p:mfd2}) so that $H\models \phi_M$.
We have that $N$ also admits a $C^1$--smooth atlas consisting of
$k$ diffeomorphic copies of $B^n(k+1)$ such that the union of the images in $\mathbb{R}^{2n+1}$ is smoothly diffeomorphic to $M$ by Lemma~\ref{lem:c-infty-ck}. This implies that $M$ and $N$ are $C^1$-- diffeomorphic (and hence, $C^\infty$--diffeomorphic). Theorem~\ref{thm:mfd} is thus proved.

\section{Regularity detection}\label{s:reg}
In this section, we establish Theorem~\ref{thm:reg}.
The basic strategy is to employ a certain non-smoothability condition in~\cite{Harrison1979} that constructs a homeomorphism
of a high-dimensional manifold with specified regularity.
Let us write
\[\mathbb{R}^n_+:=\{(x_1,\ldots,x_n)\in\mathbb{R}^n\mid x_n\ge0\}.\]
We denote by $D^2$ the closed unit 2-disk centered at the origin in $\mathbb{R}^2$.
We have the smooth standard diffeomorphism
\[
\gamma: \mathbb{R}^n_+\times S^1\cup \mathbb{R}^{n-1}\times D^2
\to \mathbb{R}^{n+1}
\]
extending the identity on $\mathbb{R}^{n-1}\times D^2$. Namely, 
we set
\[\gamma(x_1,\ldots,x_n,\theta)
:=(x_1,\dots,x_{n-1},(x_n+1)\cos\theta,(x_n+1)\sin\theta)\]
on $\mathbb{R}^n_+\times S^1$.

\begin{defn}\label{defn:harrison}
Let $\{r_k\}_{k\in\omega}$ be a positive real sequence,
and let $\{P_k\}_{k\in\omega}$ be a sequence of points converging to some point $P$, all of which belong to $\inte \mathbb{R}^{n}_+$.
Suppose we have some $f\in\Homeo_0(\mathbb{R}^{n+1})$ supported in the disjoint union of tubular neighborhoods of the collection $\{\gamma(P_k\times S^1)\}_{k\in\omega}$. 
If for all $k\in\mathbb{Z}_+$ and $\theta\in S^1$ we have \[f\circ \gamma(P_k,\theta)=\gamma(P_k,\theta+r_k),\] 
then we say $f$ is a \emph{Harrison homeomorphism} with respect to
$\{(P_k,r_k)\}_{k\in\omega}$.
\end{defn}
Intuitively, $f$ is making infinitely many wiggles along tubular neighborhoods $N_k$, which would become an obstruction for regularity. More precisely, we can summarize the result of Harrison and the key steps of her proof as follows:

\begin{thm}[\cite{Harrison1979}]\label{thm:harrison}
Let $n\ge1$ be an integer, and let $r>0$ be a real number.
\begin{enumerate}
    \item\label{p:harrison1} If a Harrison homeomorphism of $\mathbb{R}^{n+1}$ with respect to some $\{(P_k,r_k)\}_{k\in\omega}$ is topologically conjugate to a $C^r$ diffeomorphism, then \[\sum_k r_k^{n/r}<\infty.\]
    \item\label{p:harrison2} If for each $k$ there exists a
    collection of pairwise disjoint cubes $C_k$ in $\mathbb{R}^n_+$ of radius $a_k$ converging to a point,
    and if $r_k/a_k^r\to 0$, 
    then there exists a (standard volume-preserving) Harrison homeomorphism 
    $f\in \Diff^r_{0,\vol}(\mathbb{R}^{n+1})$
    with respect to $\{(\mathrm{center}(C_k),r_k)\}_{k\in\omega}$.
    \item If \[\sum_kr_k^{n/r}<\infty,\]
    and if $\{P_k\}_{k\in\omega}$ is an arbitrary sequence of points in $\mathbb{R}^n_+$ converging to a point, 
    then there exists a Harrison homeomorphism $f$ of $\bR^{n+1}$ with respect to 
    $\{(P_k,r_k)\}_{k\in\omega}$ that is topologically conjugate to a $C^r$ diffeomorphism.
\end{enumerate}
\end{thm}

\begin{rem}\label{rem:harrison}
\begin{enumerate}
\item Harrison's Theorem above holds even for $0<r<1$; in this case, the $C^r$ diffeomorphism in the theorem should be read as a $C^r$--H\"older homeomorphism, and similarly for $\Diff^r$.
\item   The case $n=3$ is missing in~\cite{Harrison1979}; this omission was due to that the stable homeomorphism theorem was not known at that time in dimension four. As this latter theorem is now known, all dimensions are covered.
\item\label{p:harrison-rem2}
If $\{a_k\}_{k\in\omega}$ is a sequence of real numbers satisfying \[\sum_k a_k^n<\infty,\] then one can find $\{C_k\}_{k\in\omega}$ satisfying the hypothesis of part (\ref{p:harrison2}) above; this is the content of the Packing Lemma (\cite[Lemma 2]{Harrison1979} and \cite{MM1968}). Furthermore, it is clear from the proof of the Packing Lemma that as long as $\{a_k\}_{k\in\omega}$ is definable in the second order arithmetic, so is the sequence $\{C_k\}_{k\in\omega}$.
\end{enumerate}
\end{rem}

Let us fix an arbitrary rational number $q>0$,
and set
\[
r_k:=\left(k\log^2(k+2)\right)^{-q/n}.\]
Note that \[\sum_k r_k^{n/r}<\infty\] if and only if $r\le q$.
Furthermore, the sequence
\[
a_k:=\left(k\log^{1.5}(k+2)\right)^{-1/n}\]
satisfies \[r_k/a_k^q\longrightarrow 0 \quad\textrm{and}\quad \sum_k a_k^n<\infty.\]
For $q\ge1$, we let a formula $\rho_{M,q}(g)$ express the following.
\begin{enumerate}[(a)]
    \item\label{p:rhomq}
    There exists a $C^1$--embedding
    $
    u\co \mathbb{R}^{n+1}\to M^{n+1}$,
    expressed by Proposition~\ref{prop:smooth-ball}.
    \item 
    There is a definable sequence of cubes $\{C_k\}_{k\in\omega}$ with radii $\{a_k\}_{k\in\omega}$
    satisfying the hypotheses of Theorem~\ref{thm:harrison} (\ref{p:harrison2}), with $q$ in place of $r$; see  Remark~\ref{rem:harrison} (\ref{p:harrison-rem2}).
     Thus, we obtain a corresponding Harrison homeomorphism $f_0$  of $\mathbb{R}^{n+1}$.
    \item The element $g$ coincides with $u\circ f_0\circ u^{-1}$ on $u(\mathbb{R}^{n+1})$, and is the identity outside.
\end{enumerate}
In case $q<1$, we replace (\ref{p:rhomq}) by:
\begin{enumerate}[(a)]
    \item[(a)'] There exists a topological embedding
    $
    u\co \mathbb{R}^{n+1}\to M^{n+1}$,
    expressed by Lemma~\ref{lem:balls}.
    \end{enumerate}    
\begin{proof}[Proof of Theorem~\ref{thm:reg}]
For part~(\ref{p:reg1}), we can choose a parametrization $u\co \mathbb{R}^{n+1}
\longrightarrow M$ as above, and then pick a (standard volume-preserving) Harrison homeomorphism given by Theorem~\ref{thm:harrison} (\ref{p:harrison2}). Since we can choose $u$ to be smooth, 
the map $u f_0 u^{-1}$ extends to a $C^q$--diffeomorphism isotopic to the identity.
Furthermore, this map preserves the volume form after a suitable smooth conjugation, by Moser's theorem.

For part~(\ref{p:reg2}), suppose we have some 
$g\in G$ satisfies $\rho_{M,q}(g)$ for some $s>q$. In other words, we have a Harrison homeomorphism $h\in\Homeo_+(\mathbb{R}^{n+1})$ with the parameters $\{r_k\}_{k\in\omega}$ as given above. We also have a $C^1$ embedding
\[
u\co \mathbb{R}^{n+1}\longrightarrow M\]
such that some element $g\in G\le \Diff^s(M)$ 
extends $u \circ h \circ u^{-1}$ by identity.
We may assume that $u(\mathbb{R}^{n+1})$ is contained in some smooth chart $v\co \mathbb{R}^{n+1}\to M$.
Then we have
\[
v^{-1}\circ u\circ h\circ u^{-1}\circ v\in\Diff^s(\mathbb{R}^{n+1}).\]
In particular, the map $h$ is conjugated by a $C^1$ map $v^{-1}\circ u$ to a $C^s$ diffeomorphism; this contradicts Theorem~\ref{thm:harrison} (\ref{p:harrison1}).
\end{proof}

A simple variation of the above argument (using $\rho_{M,1/2}$, for instance) proves that in Main Theorem, if one of $r$ or $s$ is zero then the other is also zero;
the content of Section~\ref{s:mfd} is not needed in this case.
Combined with Theorem~\ref{thm:reg}, we obtain the part ``$r=s$'' of the conclusion in Main Theorem for the case $\dim M, \dim N \ge2$.

With regard to the Main Theorem, it only remains for us to prove the case that $M$ and $N$ are copies of $S^1$.
The basic strategy for this case $\dim M=1$ is to mimic an idea in~\cite{KK2020crit}, which is to construct diffeomorphisms
of the interval and of the circle of specified regularity. 
We first note the following consequence of Theorem 2.9 in~\cite{KK2020crit}, which asserts that a regularity gives an upper asymptotic bound for the density of the set $A_f$ of ``indices with big displacements'' of increasing powers $f^{N_i}$.

\begin{lem}\label{lem:delta-fast-1}
   Let $r\in\R_{\geq 1}$ be written as $r=k+\tau$, where $k\in\N$ and $\tau\in [0,1)$, and let
   $\delta>0$.
   Suppose that:
   \begin{enumerate}
       \item $f$ is a $C^r$ diffeomorphism of $I$.
       \item $f$ is supported on an infinite collection
       $\{J_i\subseteq I\}_{i\in\omega}$ of disjoint intervals
       such that $\sup J_i<\inf J_{i+1}$ for each $i$.
       \item $\{N_i\}_{i\in\omega}\subseteq \N$ is a sequence of natural numbers such that
       \[\sup_i \frac{N_i}{i^{k-1+\tau}}<\infty.\]
   \end{enumerate}
   Write $i\in A_f$ if there exists a point $x\in J_i$ such that
   \[\frac{|f^{N_i}(x)-x|}{|J_i|}\geq \delta.\] Then
   \[\lim_{n\to\infty}\frac{|A_f\cap [0,n]|}{n}= 0.\]
\end{lem}

On the other hand, there exists a diffeomorphism of each specified regularity having an optimal displacements in the following sense; the proof is given in Theorem 2.14 in~\cite{KK2020crit}:
\begin{lem}\label{lem:delta-fast-2}
    Let $r\in\R_{\geq 1}$ be written as $r=k+\tau$, where $k\in\N$ and $\tau\in [0,1)$,
    let $\delta>0$, and let $\{J_i\subseteq I\}_{i\in\omega}$ be a collection
    of disjoint intervals
       such that $\sup J_i<\inf J_{i+1}$ for each $i$. If $\{N_i\}_{i\in\omega}\subseteq \N$
       is a sequence of natural numbers such that \[\inf_i N_i\cdot |J_i|^{k-1+\tau}\geq 1,\]
       then there exists a $C^r$ diffeomorphism $f$ of $I$ such that:
       \begin{enumerate}
           \item $f$ is supported on $\bigcup_i J_i$.
           \item For all $i$, there exists a point $x\in J_i$ such that
           \[\frac{|f^{N_i}(x)-x|}{|J_i|}\geq \delta.\]
       \end{enumerate}
\end{lem}

The case that $r=0$ would require separate ingredients. 
For a homeomorphism $f$, we write $\supp f$ for the complement of $\fix f$. In~\cite{KK2018JT}, the first two authors proved the \emph{abt-Lemma}, which asserts that for all compact connected one-manifold $M$ and for all $a,b,t\in\Diff^1_+(M)$
satisfying $\supp a\cap\supp b=\varnothing$,
one has $\form{a,b,t}\not\cong \mathbb{Z}^2\ast\mathbb{Z}$. The proof for this lemma has the following consequence, as noted in~\cite{KK2021book}.

\begin{lem}[{\cite[Lemma 3.8]{KK2018JT}; also see the proof of \cite[Lemma 4.3.11]{KK2021book}}]\label{lem:abt-var}
Let \[\{a,b,c,d\}\] be orientation-preserving homeomorphisms of the interval $I$ satisfying  $$\supp a\cap \supp b = \varnothing=\supp c\cap \supp d.$$
We let $S:=\supp a\cup \supp b\cup \supp c\cup \supp d$, and
$$
u:=\left[
\left[c,bdb^{-1}
\right],a
\right].
$$
\begin{enumerate}
\item If $\cl u \subseteq S$,
then either $u=1$ or  $\form{a,b,c,d}$ contains a copy of  $\mathbb{Z}\wr\mathbb{Z}$.
\item If we further assume that $a,b,c,d\in\Diff^1_+(I)$, then we have
$\cl u \subseteq S$.
\end{enumerate}
\end{lem}

\begin{proof}[Proof of the Main Theorem, when $\dim M=1$]
We may let $M=N=S^1$. 
Let us assume $r<s$ for contradiction. 
Consider first the case $r\ge1$. 
For each group between $\Diff^\infty(S^1)$ and $\Diff^1(S^1)$, we have a definable smooth parametrized interval in $S^1$; this can be obtained from the Bochner--Montgomery Theorem and the same argument as in  Proposition~\ref{prop:smooth-ball}. We fix such an interval \[\phi\colon I\longrightarrow S^1.\] If we equip $S^1$ with the usual angular metric and $I$ with the usual Euclidean metric, it follows from the fact that $\phi$ is smooth that the induced angular metric on $\phi(I)\subseteq S^1$ is bi-Lipschitz equivalent to the Euclidean metric on $I$.

We use Lemma~\ref{lem:delta-fast-1} and Lemma~\ref{lem:delta-fast-2}. 
Let us fix a rational $q$ between $r$ and $s$ and find a formula that is satisfied by an element of $G\ge\Diff^r(S^1)$ but by no element of $H\le\Diff^s(S^1)$, which expresses that such an element is a $C^{q}$ diffeomorphism of the circle but is not smoother. We encode a sequence of disjoint intervals $\{J_i\subseteq I\}_{i\in\omega}$ with rational endpoints and satisfying $\sup J_i<\inf J_{i+1}$ for all $i$, by a single definable real number. We can then construct a formula $\rho_{q}$ as in Theorem~\ref{thm:reg} (though the details are simpler in this case) which is satisfied by some element of $G$ but no element of $H$.

Let us now assume $r=0$ and $s\ge 1$, so that $G=\Homeo(S^1)$ and $H\ge\Diff^1(S^1)$. 
In AGAPE language, we can find a sentence $\rho$ expressing all of the following:
\begin{itemize}
\item There exist elements $a,b,c,d$ such that the union of the extended supports is a proper subset of $S^1$;
\item The hypothesis of Lemma~\ref{lem:abt-var} is met for $a,b,c,d$; here, we observe that orientation preserving elements and the support of an element are definable.
\item The element $u$ in the lemma is nontrivial, and $\cl \suppe u\not\subseteq S$.
\end{itemize}
By part (2) of the lemma, the sentence $\rho$ cannot hold in $H$.
On the other hand, there is a tuple $(a,b,c,d)$ in $\Homeo_+(S^1)$ such that the hypothesis of the lemma holds, such that $u\ne 1$ and such that $\form{a,b,c,d}\cong\mathbb{Z}^2\ast \mathbb{Z}^2$. Since this latter group cannot contain $\mathbb{Z}\wr\mathbb{Z}$, we see from part (1) that $\cl \suppe u\not\subseteq S$.
This shows that $G\not\models\rho$, completing the proof.
\end{proof}

\section*{Acknowledgements}
The first and the third author are supported by Mid-Career Researcher Program (RS-2023-00278510) through the National Research Foundation funded by the government of Korea.
The first and the third authors are also supported by KIAS Individual Grants (MG073601 and MG084001, respectively) at Korea Institute for Advanced Study
and by the Samsung Science and Technology Foundation under Project Number SSTF-BA1301-51. The second author is partially supported by NSF Grants DMS-2002596
and DMS-2349814, and Simons Foundation International Grant SFI-MPS-SFM-00005890. The authors thank J.~La and  C.~Rosendal for helpful discussions. The authors thank M.~Bestvina for asking
whether Proposition~\ref{prop:compact-noncompact} holds.

\bibliographystyle{plain}
\bibliography{ref}
\end{document}